\documentclass[preprint]{article}%
\usepackage{graphicx}
\usepackage{amsmath,amsxtra,amssymb,latexsym, amscd,amsthm}
\usepackage[mathscr]{eucal}
\newtheorem{thm}{Theorem}[section]
\newtheorem{cor}{Corollary}[section]
\newtheorem{lem}{Lemma}[section]
\newtheorem{prop}{Proposition}[section]
\theoremstyle{definition}

\theoremstyle{remark}
\newtheorem{rem}{Remark}[section]
\numberwithin{equation}{section}
\def\ind{{\rm 1\hspace{-0.90ex}1}}

\begin{document}

\title{Gaussian lower bounds for the density via Malliavin calculus
}
\author{Nguyen Tien Dung\thanks{Department of Mathematics, VNU University of Science, Vietnam National University, Hanoi, 334 Nguyen
Trai, Thanh Xuan, Hanoi, 084 Vietnam. Email: dung@hus.edu.vn}\,\,\footnote{Department of Mathematics, FPT University, Hoa Lac High Tech Park, Hanoi, Vietnam. Email: dungnt@fpt.edu.vn}
}

\date{\today}          

\maketitle
\begin{abstract} In this paper, based on a known formula, we use a simple idea to get a new representation for the density of Malliavin differentiable random variables. This new representation is particularly useful for finding lower bounds for the density.
\end{abstract}
\section{Introduction}
In this paper, we use the techniques of Malliavin calculus to investigate the density of Malliavin differentiable random variables. In particular, we are going to focus on the problem of finding a Gaussian lower bound for the density. This problem was first discussed by Kusuoka \& Stroock \cite{Kusuoka1987}, and up to date, it is still a subject that is worth studying.
In the last two decades, there have been several papers devoted to the study of densities by means of Malliavin calculus. Among others, we mention the works \cite{Malliavin2009,Enualart2004} and the references therein for sufficient conditions for a random variable to has a density bounded from below. Another fruitful contribution is Nourdin \& Viens' density formula, Theorem 3.1 in \cite{Nourdin2009},  that can be restated as follows.
\begin{prop}\label{prop:denspoiss}
 Let $F\in\mathbb{D}^{1,2}$ be such that $E[F]=0.$ We define the random variable
 \begin{equation}\label{ijrm}
G_F:=\langle DF,-DL^{-1}F\rangle_{\mathfrak{H}}
\end{equation}
and the function $g_F(x):=E[G_F|F=x].$ Then, the law of $F$ has a density $\rho_F$  with respect to the Lebesgue measure if and only if $g_F(F)>0$ a.s.
 In this case $\mathrm{supp}\,\rho_F$ is  a closed interval of $\mathbb{R}$ containing $0$  and we have, for almost all $x\in \mathrm{supp}\,\rho_F:$
\begin{equation}\label{densityCIR05}
\rho_{F}(x)=\frac{E|F|}{2g_F(x)}\exp\left(-\int_0^x\frac{z}{g_F(z)}dz\right).
\end{equation}
\end{prop}The definition of the Malliavin derivative $D$ and the operator $L^{-1}$ will be given in Section \ref{uufk}. 
The formula (\ref{densityCIR05}) has been effectively applied to various stochastic equations (see e.g. \cite{DungPG2016} and the references therein). However, its use requires both lower and upper bounds of $G_F.$  In fact, if $\sigma^2_{min}\leq G_F\leq \sigma^2_{max}\,\,a.s.$ then the density of $F$
satisfies
$$\frac{E|F|}{2\sigma^2_{max}}\exp\left(-\frac{x^2}{2\sigma^2_{min}}\right)\leq \rho_F(x)\leq \frac{E|F|}{2\sigma^2_{min}}\exp\left(-\frac{x^2}{2\sigma^2_{max}}\right),\,\,x\in \mathbb{R}.$$
The aim of the present paper is to answer the following question: Can we prove a Gaussian lower bound for the density if we only suppose
that $G_F\geq\sigma^2_{min}?$ (similarly, a Gaussian upper bound if  $0<G_F\leq\sigma^2_{max}$).

The rest of this article is organized as follows. In Section \ref{uufk}, we briefly recall some of the relevant elements of the Malliavin calculus. In Section \ref{hje3}, based on a density formula provided in \cite{nualartm2}, we use a simple idea to obtain a new representation formula for densities. As a consequence, under some additional assumptions, we are able to give an affirmative answer to the above question. In Section \ref{hje3a}, we provide some examples to illustrate the applicability of our abstract results.
\section{Malliavin Calculus}\label{uufk}
Let us recall some elements of Malliavin calculus that we need in order to perform our proofs (for more details see \cite{nualartm2}). Suppose that $\mathfrak{H}$ is a real separable Hilbert space with scalar product denoted by $\langle.,.\rangle_\mathfrak{H}.$ We denote by $W = \{W(h) : h \in \mathfrak{H}\}$ an isonormal Gaussian process defined in a complete probability space $(\Omega,\mathcal{F},P),$ $\mathcal{F}$ is the $\sigma$-field generated by $W.$ Let $\mathcal{S}$ be the set of all smooth cylindrical random variables of the form
\begin{equation}\label{ro}
F=f(W(h_1),...,W(h_n)),
\end{equation}
where $n\in \mathbb{N}, f\in C_b^\infty(\mathbb{R}^n)$ the set of bounded and infinitely differentiable functions with bounded partial derivatives, $h_1,...,h_n\in \mathfrak{H}.$ If $F$ has the form (\ref{ro}), we define its Malliavin derivative with respect to $W$ as the element of $L^2(\Omega,\mathfrak{H})$ given by
$$DF=\sum\limits_{k=1}^n \frac{\partial f}{\partial x_k}(W(h_1),...,W(h_n)) h_k.$$
More generally, we can define the $k$th order derivative $D^kF\in L^2(\Omega, \mathfrak{H}^{\otimes k})$ by iterating the derivative operator $k$ times. For any integer $k\geq 1$ and any $p\geq 1,$ we denote by $\mathbb{D}^{k,p}$ the closure of $\mathcal{S}$ with respect to the norm
$$\|F\|^p_{k,p}:=E|F|^p+\sum\limits_{i=1}^kE\|D^i F\|^p_{\mathfrak{H}^{\otimes i}}.$$
An important operator in the Malliavin calculus theory is the divergence operator $\delta,$ it is the adjoint of the derivative operator $D$ characterized by
$$E\langle DF,u\rangle_{\mathfrak{H}}=E[F\delta(u)]$$
for any $F\in \mathcal{S}$ and $u\in L^2(\Omega,\mathfrak{H}).$ The domain of $\delta$ is the set of all processes $u\in L^2(\Omega,\mathfrak{H})$ such that
$$E|\langle DF,u\rangle_{\mathfrak{H}}|\leq C(u)\|F\|_{L^2(\Omega)},$$
where $C(u)$ is some positive constant depending only on $u.$ Let $F\in \mathbb{D}^{1,2}$ and $u\in Dom\,\delta$ such that  $Fu\in L^2(\Omega,\mathfrak{H}).$ Then $Fu\in Dom\,\delta$ and we have the following relation
\begin{equation}\label{hj44}
\delta(Fu)=F\delta(u)-\left\langle DF,u\right\rangle _{\mathfrak{H}},
\end{equation}
provided the right-hand side is square integrable.

It is known that any random variable $F$ in $L^2(\Omega, \mathcal{F},P)$ can be expanded into an orthogonal sum of its Wiener chaos:
$$F=\sum\limits_{n=0}^\infty J_nF,$$
where $J_0F=E(F)$ and $J_n$ denotes the projection onto the $n$th Wiener chaos. From this chaos expansion one may define the Ornstein-Uhlenbeck operator $L$ by $LF=\sum\limits_{n=0}^\infty-n J_nF$ when $F\in \mathbb{D}^{2,2}$ and its pseudo-inverse by $L^{-1}F=\sum\limits_{n=1}^\infty-\frac{1}{n} J_nF.$ Note that, for any $F\in L^2(\Omega),$ we have $L^{-1}F\in Dom\,L$ and $LL^{-1}F=L^{-1}LF=F-E[F].$ Moreover, the operators $D,\delta$ and $L$ satisfy the following relationship: $F\in Dom\,L$ if and only if $F\in \mathbb{D}^{2,2}$ and, in this case,
\begin{equation}\label{hj4}
\delta D F=-LF.
\end{equation}

\section{Representation and lower bounds for the density}\label{hje3}
This section contains our abstract results, we first provide a representation formula for densities.
\begin{prop}\label{uhuy}
Let $F\in \mathbb{D}^{1,2}$ and $u:\Omega \rightarrow \mathfrak{H}$, and
suppose that $\left\langle DF,u\right\rangle _{\mathfrak{H}}\neq 0$ a.s. and
$\frac{u}{\left\langle DF,u\right\rangle _{\mathfrak{H}}}$ belongs to the domain of $\delta.$ Then the law of $F$ has a continuous density
given by
\begin{equation}\label{kofkl}
\rho_F(x)=\rho_F(a)\exp\left(-\int_a^xw(z)dz\right),\,\,\,x\in \mathrm{supp}\,\rho_F,
\end{equation}
where $a$ is a point in the interior of $\mathrm{supp}\,\rho_F$ and
$$w(z):=E\left[\delta \left(
\frac{u}{\left\langle DF,u\right\rangle _{\mathfrak{H}}}\right)\big| F=z \right].$$
\end{prop}
\begin{proof}According to Exercise 2.1.3 in \cite{nualartm2}, the law of $F$ has a continuous density
given by
\begin{equation}\label{Fmla2}
\rho_{F}\left( x\right) =E\left[ \mathbf{1}_{\left\{ F>x\right\} }\delta \left(
\frac{u}{\left\langle DF,u\right\rangle _{\mathfrak{H}}}\right) \right],\,\,\, x\in \mathrm{supp}\,\rho_F.
\end{equation}
Note that the proof of (\ref{Fmla2}) is similar to that of Proposition 2.1.1 in \cite{nualartm2}. Since $F\in\mathbb{D}^{1,2},$ this implies that $\mathrm{supp}\,\rho_F$ is  a closed interval of $\mathbb{R}$ (see Proposition 2.1.7 in \cite{nualartm2}): $\mathrm{supp}\,\rho_F=[\alpha,\beta]$ with $-\infty\leq \alpha<\beta\leq \infty.$ It follows from (\ref{Fmla2}) that
\begin{align*}
\rho_{F}\left( x\right)&=E\left[ \mathbf{1}_{\left\{ F>x\right\} }E\left[ \delta \left(
\frac{u}{\left\langle DF,u\right\rangle _{\mathfrak{H}}}\right)\big| F\right]\right]\\
&=E\left[ \mathbf{1}_{\left\{ F>x\right\} }w_F(F)\right]\\
&=\int_x^\beta w_F(y)\rho_{F}(y)dy.
\end{align*}
Let $a$ be a point in the interior of $\mathrm{supp}\,\rho_F.$  Solving the above equation with initial condition $\rho_F(a)$ gives us (\ref{kofkl}). This completes the proof.
\end{proof}
A general representation like the above conveys no meaning unless provided at least a way to use it. The following corollary provides such a way.
\begin{cor}\label{vik}Let $F\in\mathbb{D}^{2,4}$ be such that $E[F]=0$ and $G_F$ be the random variable defined by (\ref{ijrm}). Assume that $G_F\neq 0$ a.s. and the random variables $\frac{F}{G_F}$ and $\frac{1}{G_F^2}\langle DG_F,-DL^{-1}F\rangle_{\mathfrak{H}}$ belong to $L^2(\Omega).$ Then the law of $F$ has a continuous density given by
\begin{equation}\label{kofkl2}
\rho_F(x)= \rho_F(0)\exp\left(-\int_0^x h_F(z)dz\right)\exp\left(-\int_0^x w_F(z)dz\right),\,\,\,x\in\mathrm{supp}\,\rho_F,
\end{equation}
where the functions $w_F$ and $h_F$ are defined by
$$w_F(z):=E\left[\frac{F}{G_F}\big| F=z\right],\,\,\,h_F(z):=E\left[\frac{1}{G_F^2}\langle DG_F,-DL^{-1}F\rangle_{\mathfrak{H}}\big| F=z\right].$$
\end{cor}
\begin{proof}Since $E[F]=0,$ this implies that $\alpha<0<\beta$ and hence, we can take $a=0$ in Theorem \ref{uhuy}. On the other hand, we choose $u=-DL^{-1}F.$ By the relation (\ref{hj4}) we have
$$\delta(u)=-\delta(DL^{-1}F)=LL^{-1}F=E-E[F]=F.$$
The conditions on $F$ and $G_F$ allow us to use the relation (\ref{hj44}) and we obtain
\begin{align*}
\delta \left(\frac{u}{\left\langle DF,u\right\rangle _{\mathfrak{H}}}\right)&=\delta \left(\frac{u}{\left\langle DF,-DL^{-1}F\right\rangle _{\mathfrak{H}}}\right)\\
&=\delta \left(\frac{u}{G_F}\right)\\
&=\frac{\delta \left(u\right)}{G_F}+\frac{1}{G_F^2}\langle DG_F,u\rangle_{\mathfrak{H}}\\
&=\frac{F}{G_F}+\frac{1}{G_F^2}\langle DG_F,-DL^{-1}F\rangle_{\mathfrak{H}}.
\end{align*}
Hence, we obtain $w(F)=w_F(F)+h_F(F).$ Inserting this relation into (\ref{kofkl}) gives us (\ref{kofkl2}). This completes the proof.
\end{proof}
We now are ready to provide Gaussian lower bounds for the density.
\begin{thm}\label{oklf}Let $F\in\mathbb{D}^{2,4}$ be such that $E[F]=0.$ Suppose that $G_F\geq \sigma^2_{min}$ a.s. for some deterministic constant  $\sigma_{min}\neq 0.$ Then, the density of $F$ exists and satisfies
\begin{equation}\label{bvj1}
\rho_F(x)\geq \rho_F(0)\exp\left(-\int_0^x h_F(z)dz\right)\exp\left(-\frac{x^2}{2\sigma^2_{min}}\right),\,\,\,x\in \mathbb{R}.
\end{equation}
Moreover, if for some real number $m_1,$ $h_F(F)\geq m_1\,\,a.s.$  then
\begin{equation}\label{bvj2}\rho_F(x)\geq \rho_F(0)\exp\left(-\frac{x^2}{2\sigma^2_{min}}-m_1x\right),\,\,\,x\leq 0.
\end{equation}
If for some real number $m_2,$ $h_F(F)\leq m_2\,\,a.s.$ then
\begin{equation}\label{bvj3}\rho_F(x)\geq \rho_F(0)\exp\left(-\frac{x^2}{2\sigma^2_{min}}-m_2x\right),\,\,\,x\geq 0.
\end{equation}
If for some real number $M>0,$ $|h_F(F)|\leq M\,\,a.s.$ then
\begin{equation}\label{bvj4}
\rho_F(x)\geq \rho_F(0)\exp\left(-\frac{x^2}{2\sigma^2_{min}}-M|x|\right),\,\,\,x\in \mathbb{R}.
\end{equation}
\end{thm}
\begin{proof} We first recall that the fact $G_F\geq \sigma^2_{min}$ implies $\mathrm{supp}\,\rho_F=\mathbb{R},$ see Corollary 3.3 in \cite{Nourdin2009}. 
When $x\geq 0,$ we have
\begin{align*}
-\int_0^x w_F(z)dz
&\geq -\int_0^x E\left[\frac{F}{\sigma^2_{min}}\big| F=z\right]dz\\
&= -\int_0^x \frac{z}{\sigma^2_{min}}dz=-\frac{x^2}{2\sigma^2_{min}}.
\end{align*}
Similarly, when $x\leq 0,$ we also have
$$-\int_0^x w_F(z)dz=\int^0_x E\left[\frac{F}{G_F}\big| F=z\right]dz\geq  \int^0_x E\left[\frac{F}{\sigma^2_{min}}\big| F=z\right]dz= -\frac{x^2}{2\sigma^2_{min}}.$$
Thus (\ref{bvj1}) is verified for all $x\in \mathbb{R}.$ The proof of (\ref{bvj2}), (\ref{bvj3}) and (\ref{bvj4})  is straightforward, so we omit it.
\end{proof}
\begin{rem}
Similarly, if $0< G_F\leq \sigma^2_{max}$ a.s. we also have an upper bound for the density that reads
$$\rho_F(x)\leq \rho_F(0)\exp\left(-\int_0^x h_F(z)dz\right)\exp\left(-\frac{x^2}{2\sigma^2_{max}}\right),\,\,\,x\in\mathrm{supp}\,\rho_F.$$
However, because of appearance of $G_F$ in the denominators, it will be non-trivial to check the square integrable property of $\frac{F}{G_F}$ and $\frac{1}{G_F^2}\langle DG_F,-DL^{-1}F\rangle_{\mathfrak{H}}$ and the boundedness of $h_F.$ That is why we only provide lower bounds as in Theorem \ref{oklf}. To evaluate an upper bound, a popular method is to use the formula (\ref{Fmla2}) with $u=DF$ . The reader can consult Proposition 2.1.2 in \cite{nualartm2} for such a evaluation.
\end{rem}
\begin{rem}
If the random variable $G_F$ satisfies $\sigma^2_{min}\leq G_F\leq \sigma^2_{max}$ a.s. the formula (\ref{kofkl2}) will provide us lower and upper bounds for the density. However, in this case, we should use the formula (\ref{densityCIR05}) to get Gaussian estimates for the density because Proposition \ref{prop:denspoiss} only requires $F\in \mathbb{D}^{1,2}.$
\end{rem}
We end up this section by providing a variant of density formula (\ref{kofkl2}) which can be of interest for the readers who are not used to working with the Ornstein-Uhlenbeck operator. Let $(W_t)_{t\in[0,T]}$ be a standard Brownian motion defined on a complete probability space $(\Omega,\mathcal{F},\mathbb{F},P)$, where $\mathbb{F}=(\mathcal{F}_t)_{t\in [0,T]}$ is a natural filtration generated by $W.$ Now Malliavin derivative operator is with respect to $W$ and $\mathfrak{H}=L^2[0,T].$ We consider the stochastic process $u_s:=E[D_sF|\mathcal{F}_s].$ Then, by the Clark-Ocone formula we have
$$\delta(u)=\int_0^TE[D_sF|\mathcal{F}_s]dW_s=F-E[F].$$
Hence, with the exact proof of Corollary \ref{vik}, we obtain the following.
\begin{thm}\label{9hj3}Let $F\in\mathbb{D}^{2,4}$ be such that $E[F]=0.$ Define the random variable
$$\Phi_F:=\int_0^TD_sFE[D_sF|\mathcal{F}_s]ds.$$
 Assume that $\Phi_F\neq 0$ a.s. and the random variables $\frac{F}{\Phi_F}$ and $\frac{1}{\Phi_F^2}\int_0^TD_s\Phi_FE[D_sF|\mathcal{F}_s]ds$ belong to $L^2(\Omega).$ Then the law of $F$ has a continuous density given by
\begin{equation}\label{krkl2}
\rho_F(x)= \rho_F(0)\exp\left(-\int_0^x \overline{h}_F(z)dz\right)\exp\left(-\int_0^x \overline{w}_F(z)dz\right),\,\,\,x\in\mathrm{supp}\,\rho_F,
\end{equation}
where the functions $\overline{w}_F$ and $\overline{h}_F$ are defined by
$$\overline{w}_F(z):=E\left[\frac{F}{\Phi_F}\big| F=z\right],\,\,\,\overline{h}_F(z):=E\left[\frac{1}{\Phi_F^2}\int_0^TD_s\Phi_FE[D_sF|\mathcal{F}_s]ds\big| F=z\right].$$
\end{thm}
\begin{rem}The conclusion of Theorem \ref{oklf} still holds true if we replace $G_F$ by $\Phi_F$ and $h_F$ by $\overline{h}_F.$
\end{rem}
\begin{rem} The following problem will be interesting to investigate: Find other choices for $u$ in Proposition \ref{uhuy}.

\end{rem}

\section{Examples}\label{hje3a} In this section, we provide some examples to illustrate the applicability of our abstract results. 
\subsection{Additive functional of Gaussian processes}
Let $(X_t)_{t\in [0,T]}$ be a centered Gaussian process with continuous paths. It is known from Section 3.2.2 in \cite{Nourdin2009} that the Gaussian space generated by $X$ can be identified with an isonormal Gaussian process of the type $X=\{X(h):h\in \mathfrak{H}\},$ where the real and separable Hilbert space $\mathfrak{H}$ is defined as follows: (i) denote by $\mathcal{E}$ the set of all $\mathbb{R}$-valued step functions on $[0, T],$ (ii) define $\mathfrak{H}$ as the Hilbert space obtained by closing $\mathcal{E}$ with respect to the scalar product
$$\left\langle \ind_{[0,s]},\ind_{[0,t]}\right\rangle_{\mathfrak{H}}=E(X_sX_t).$$
In particular, with such a notation, we identify $X_t$ with $X(\ind_{[0,t]}).$ We now consider the functional
\begin{equation}
Y_T:=\int_0^T f(X_s)ds-\int_0^T E[f(X_s)]ds.
\end{equation}
The density of $Y_T$ has been discussed by Nourdin and Viens, see Proposition 3.10 in \cite{Nourdin2009}. In order to be able to obtain Gaussian estimates, they require the condition $c\leq f'(x)\leq C$ for all $x\in\mathbb{ R}$ and for some $C,c>0.$ Our Theorem \ref{oklf} allows us to address the case, where $f'(x)$ is not bounded above, and we obtain the following.
\begin{thm}Assume that $E[X_sX_v]\geq 0$ for all $s,v\in [0,T],$ and $f:\mathbb{R}\to \mathbb{R}$ is a twice differentiable function satisfying $|f'(x)|\geq c$  for all $x\in \mathbb{R}.$ Then, the random variable $Y_T$ admits a density, which satisfies

\noindent(i) If $f''(x)\geq 0$ for all $x\in \mathbb{R},$ then
\begin{equation}\label{ffjk2}
\rho_{Y_T}(x)\geq \rho_{Y_T}(0)\exp\left(-\frac{x^2}{2c^2\sigma^2_{T}}\right),\,\,\,x\leq 0.
\end{equation}
\noindent(ii) If $f''(x)\leq 0$ for all $x\in \mathbb{R},$ then
\begin{equation}\label{ffjk3}\rho_{Y_T}(x)\geq \rho_{Y_T}(0)\exp\left(-\frac{x^2}{2c^2\sigma^2_{T}}\right),\,\,\,x\geq 0,
\end{equation}
where $\sigma^2_{T}:=\int_0^T\int_0^TE[X_sX_v]dsdv.$
\end{thm}
\begin{proof}We only consider the case $f'(x)\geq c$ for all $x\in \mathbb{R}$ because the case $f'(x)\leq -c$ can be treated similarly. The Malliavin derivative of $Y_T$ with respect to $X$ is given by
$$D_rY_T=\int_0^T f'(X_s)\ind_{[0,s]}(r)ds,\,\,r\in [0,T].$$
Thanks to Proposition 3.7 in \cite{Nourdin2009} we have
$$-D_rL^{-1}Y_T=\int_0^\infty e^{-u}\int_0^T E'[f'(e^{-u}X_s+\sqrt{1-e^{-2u}}X'_s)]\ind_{[0,s]}(r)dsdu,\,\,r\in [0,T].$$
and
$$G_{Y_T}=\int_0^\infty e^{-u}\int_{0}^T\int_{0}^Tf'(X_s)E'[f'(e^{-u}X_v+\sqrt{1-e^{-2u}}X'_v)]E[X_sX_v]dsdvdu,$$
where $X'$ stands for an independent copy of $X$ and $E'$ is the
expectation with respect to $X'.$ Hence, it holds that
\begin{align*}
G_{Y_T}&\geq \int_0^\infty e^{-u}\int_{0}^T\int_{0}^Tc^2E[X_sX_v]dsdvdu\\
&= c^2\int_0^T\int_0^TE[X_sX_v]dsdv=c^2\sigma^2_{T}\,\,a.s.
\end{align*}
Furthermore, we have, for $r,\theta\in [0,T],$
$$D_\theta D_rY_T=\int_0^T f''(X_s)\ind_{[0,s]}(r)\ind_{[0,s]}(\theta)ds,$$
$$-D_\theta D_rL^{-1}Y_T=\int_0^\infty e^{-2u}\int_0^T E'[f''(e^{-u}X_s+\sqrt{1-e^{-2u}}X'_s)]\ind_{[0,s]}(r)\ind_{[0,s]}(\theta)dsdu.$$
Thus, if $f''(x)\geq 0$ for all $x\in \mathbb{R},$ then $D_\theta D_rY_T\geq 0$ and $-D_\theta D_rL^{-1}Y_T\geq 0.$ As a consequence, by its definition, $h_F(F)\geq 0\,\,a.s.$ So (\ref{ffjk2}) follows from (\ref{bvj2}).

Similarly, if $f''(x)\leq 0$ for all $x\in \mathbb{R},$ then
$h_F(F)\leq 0\,\,a.s.$ and (\ref{ffjk3}) follows from (\ref{bvj3}).
This completes the proof.
\end{proof}

\subsection{SDEs with fractional noise}
We consider stochastic differential equations driven by fractional Brownian motion of the form
\begin{equation}\label{7hj3}
X_t=x_0+\int_0^tb(s,X_s)ds+\int_0^t\sigma(s,X_s)dB^H_s,\,\,t\in[0,T],
\end{equation}
where $x_0\in \mathbb{R},$ $B^H=(B^H)_{t\in [0,T]}$ is a fractional
Brownian motion (fBm) of Hurst parameter $H\in (\frac{1}{2},1)$ and
the stochastic integral is interpreted as a pathwise
Riemann-Stieltjes integral, see e.g. \cite{M. Zahle}. Recall that
$B^H$ is a centered Gaussian process and it admits the so-called
Volterra representation (see e.g. \cite{nualartm2} pp. 277-279)
\begin{equation}\label{densityCIR02}
B^H_t=\int_0^t K_H(t,s)dW_s,
\end{equation}
where $(W_t)_{t\in [0,T]}$ is a standard Brownian motion,
\[
K_H(t,s):=c_H\,s^{1/2-H}\int_s^t (u-s)^{H-\frac{3}{2}}u^{H-1/2}d u,\quad\text{$s\leq t$}
\]
and $c_H=\sqrt{\frac{H(2H-1)}{\beta(2-2H,H-1/2)}},$ where $\beta$ is the Beta function.

By different approaches, the density estimates for the solutions to
the equation (\ref{7hj3}) have been recently obtained in
\cite{Besalu2016,DungPG2014}. In both these two papers, the authors
require $c\leq |\sigma(t,x)|\leq C$ for all $(t,x)\in
[0,T]\times\mathbb{ R}$ and for some $C,c>0.$ When $H=\frac{1}{2},$
$B^H$ reduces to a Brownian motion and in this case, Nualart
\cite[Theorem 2.3]{Enualart2004} only requires $|\sigma(t,x)|\geq c$ to get a
Gaussian lower bound. Here we are able to obtain such a similar
result for the case $H>\frac{1}{2}.$

For a differentiable function $f,$ we denote
\[
f'_1(t,x):=\frac{\partial f}{\partial t}(t,x),\quad f'_2(t,x):=\frac{\partial f}{\partial x}(t,x).
\]

\begin{lem}\label{le:malder}Suppose that $b,\sigma\in\mathcal{C}^{1,1}([0,T]\times \mathbb{R})$ and there exists a constants $c>0$ so that $|\sigma(t,x)|\geq c$ for all $(t,x)\in [0,T]\times\mathbb{ R}.$ In addition, we assume that the function
$$m(t,x):=\left(b'_2-\frac{b\sigma'_2}{\sigma}-\frac{\sigma'_1}{\sigma}\right)(t,x)$$
is bounded on $[0,T]\times\mathbb{ R}.$ Then, the Malliavin derivative of $X_t$ with respect to Brownian motion $W$ is given by
$$D_s X_t=\sigma(t,X_t)\left(\int_s^t(K_H)_1'(v,s)\exp \left(
 \int_v^tm(u,X_u)d u\right)d v\right)\ind_{[0,t]}(s).$$
\end{lem}
\begin{proof}The proof is the same as that of Lemma 5.3 in \cite{DungPG2014}. Notice that the boundedness of $m$ ensures that the equation (5.6) in \cite{DungPG2014} satisfies the global Lipschitz and linear growth conditions and hence, its solution is Malliavin differentiable.
\end{proof}

\begin{thm}\label{832}Suppose the assumption of Lemma \ref{le:malder}. In addition, we assume that there exists $M>0$ so that $|m(t,x)|,|m'_2(t,x)\sigma(t,x)|,|\sigma'_2(t,x)|\leq M$ for all $(t,x)\in[0,T]\times\mathbb{ R}.$ Then, for each $t\in (0,T],$ the density of $X_t$ exists and
$$\rho_{X_t}(x)\geq c_1\exp\left(-\frac{(x-E[X_t])^2}{2c_2t^{2H}}\right),\,\,\,x\in \mathbb{R},$$
where $c_1,c_2$ are positive constants.
\end{thm}
\begin{proof}We assume $\sigma(t,x)\geq c$, the case $\sigma(t,x)\leq -c$ can be treated similarly. Thus we always have $D_s X_t\geq 0\,\,a.s.$ For the simplicity, we write $D_s X_t=\sigma(t,X_t)\varphi(t,s),$ where
$$\varphi(t,s):=\left(\int_s^t(K_H)_1'(v,s)\exp \left( \int_v^tm(u,X_u)d u\right)d v\right)\ind_{[0,t]}(s).$$
We have
$$D_r\varphi(t,s)=\left(\int_s^t(K_H)_1'(v,s)\left(\int_v^tm'(u,X_u)\sigma(u,X_u)\varphi(u,r)d u\right)d v\exp \left(\int_v^tm(u,X_u)d u\right)d v\right)\ind_{[0,t]}(s).$$
The boundedness of $m$ yields
$$e^{-MT}K_H(t,s)\leq \varphi(t,s)\leq e^{MT}K_H(t,s),\,\,0\leq s\leq t\leq T.$$
Since $m'\sigma$ is bounded, this implies that $|\int_v^t m'(u,X_u)\sigma(u,X_u)\varphi(u,r)d u|\leq M\int_v^t e^{MT}K_H(u,r)du\leq MT e^{MT}K_H(t,r),$ and hence,
$$|D_r\varphi(t,s)|\leq MT e^{MT}K_H(t,r)\varphi(t,s),\,\,0\leq r, s\leq t\leq T.$$
We have
\begin{align*}
D_rD_s X_t&=\sigma'_2(t,X_t)D_rX_t\varphi(t,s)+\sigma(t,X_t)D_r\varphi(t,s)\\
&=\sigma'_2(t,X_t)\varphi(t,r)D_sX_t+\sigma(t,X_t)D_r\varphi(t,s)
\end{align*}
and
\begin{align*}
|D_rD_sX_t|&\leq M e^{MT}K_H(t,r)D_sX_t+\sigma(t,X_t)MT e^{MT}K_H(t,r)\varphi(t,s)\\
&=M(1+T) e^{MT}K_H(t,r)D_sX_t,\,\,0\leq r, s\leq t\leq T.
\end{align*}
Fixed $t\in (0,T].$ We now apply Theorem \ref{9hj3} to $F:=X_t-E[X_t].$ We have
\begin{align*}
\Phi_F&=\int_0^tD_sX_tE[D_sX_t|\mathcal{F}_s]ds\\
&=\int_0^t\sigma(t,X_t)\varphi(t,s)E[\sigma(t,X_t)\varphi(t,s)|\mathcal{F}_s]ds\\
&\geq c^2e^{-2MT}\int_0^t K^2_H(t,s)ds=c^2e^{-2MT} t^{2H}\,\,a.s.
\end{align*}
Furthermore, for $0\leq r\leq t,$
$$D_r\Phi_F=\int_0^tD_rD_sX_tE[D_sX_t|\mathcal{F}_s]ds+\int_0^tD_sX_tE[D_rD_sX_t|\mathcal{F}_s]\ind_{[0,s]}(r)ds,$$
which leads us to
$$|D_r\Phi_F|\leq 2M(1+T) e^{MT}K_H(t,r)\int_0^tD_sX_tE[D_sX_t|\mathcal{F}_s]ds=2M(1+T) e^{MT}K_H(t,r)\Phi_F.$$
Consequently, we deduce
\begin{align*}
\big|\int_0^tD_r\Phi_FE[D_rF|\mathcal{F}_r]dr\big|&\leq 2M(1+T) e^{MT}\int_0^t\Phi_FK_H(t,r)E[D_rF|\mathcal{F}_r]dr\\
&\leq 2M(1+T) e^{MT}\Phi_F\int_0^te^{MT}\varphi(t,r)E[D_rF|\mathcal{F}_r]dr\\
&= 2M(1+T) e^{2MT}\Phi_F\int_0^t\frac{D_rX_t}{\sigma(t,X_t)}E[D_rF|\mathcal{F}_r]dr\\
&\leq \frac{2M(1+T) e^{2MT}}{c}\Phi_F^2\,\,a.s.
\end{align*}
Recalling the definition of $\overline{h}_F,$ we obtain
$$|\overline{h}_F(F)|\leq \frac{2M(1+T) e^{2MT}}{c}\,\,a.s.$$
So we can conclude that
$$\rho_F(x)\geq \rho_F(0)\exp\left(-\frac{x^2}{2c^2e^{-2MT} t^{2H}}-\frac{2M(1+T) e^{2MT}}{c}|x|\right),\,\,\,x\in \mathbb{R}.$$
Now it is easy to see that there exist positive constants $c_1,c_2$ such that
$$\rho_F(x)\geq c_1\exp\left(-\frac{x^2}{2c_2t^{2H}}\right),\,\,\,x\in \mathbb{R}.$$
This finishes the proof because $\rho_{X_t}(x)=\rho_F(x-E[X_t]).$
\end{proof}
\begin{rem} A simple example verifying Theorem \ref{832} is when $b(t,x)=\sigma(t,x)=f(x),$ where $f\in \mathcal{C}^1(\mathbb{R})$ with bounded derivative and $|f(x)|\geq c$ for all $x\in \mathbb{R}.$

\end{rem}


\noindent {\bf Acknowledgments.} The author would like to thank the anonymous referees for their valuable comments for improving the paper. This research was funded by Viet Nam National Foundation for Science and Technology Development (NAFOSTED) under grant number 101.03-2019.08. A part of this paper was done while the author was visiting the Vietnam Institute for Advanced Study in Mathematics (VIASM). He would like to thank the VIASM for financial support and hospitality.


\begin{thebibliography}{99}

\bibitem{Besalu2016} M. Besal\'u, A. Kohatsu-Higa, S. Tindel, Gaussian-type lower bounds for the density of solutions of SDEs driven by fractional Brownian motions. {\it Ann. Probab.} 44 (2016), no. 1, 399--443.
\bibitem{DungPG2014} N.T. Dung, N. Privault, G. L. Torrisi, Gaussian estimates for the solutions of some one-dimensional stochastic equations. {\it Potential Anal} (2015) 43:289-311.

\bibitem{DungPG2016} N.T. Dung, The density of solutions to multifractional stochastic Volterra integro-differential equations. {\it Nonlinear Anal.} 130 (2016), 176--189.



\bibitem{Kusuoka1987}S. Kusuoka, D. Stroock,  Applications of the Malliavin calculus. III. {\it J. Fac. Sci. Univ. Tokyo Sect. IA Math.} 34 (1987), no. 2, 391--442.

\bibitem{Malliavin2009}P. Malliavin, E. Nualart, Density minoration of a strongly non-degenerated random variable. {\it J. Funct. Anal.} 256 (2009), no. 12, 4197--4214.

\bibitem{Nourdin2009} I. Nourdin, F.G. Viens, Density formula and concentration inequalities with Malliavin calculus. {\it Electron. J. Probab.} 14 (2009), no. 78, 2287--2309.
\bibitem{nualartm2} D. Nualart, {\em The Malliavin calculus and related topics}. Probability and its Applications. Springer-Verlag, Berlin, second
  edition, 2006.

\bibitem{Enualart2004}E. Nualart, Exponential divergence estimates and heat kernel tail. {\it C. R. Math. Acad. Sci. Paris} 338 (2004), no. 1, 77--80.

\bibitem{M. Zahle} M. Z\"{a}hle, Integration with respect to fractal functions and stochastic calculus. I. {\it Probab. Theory Related Fields,} 111:333-374, 1998.


\end{thebibliography}
\end{document}